\documentclass[10pt]{amsart}
\usepackage{amsfonts,amssymb, amsmath}
\usepackage{enumerate}
\usepackage{url}
\usepackage{graphicx}

\newtheorem{thm}{Theorem}[section]
\newtheorem{lem}[thm]{Lemma}
\newtheorem{prop}[thm]{Proposition}
\newtheorem{cor}[thm]{Corollary}

\theoremstyle{remark}
\newtheorem{remark}[thm]{Remark}
\newtheorem{defn}[thm]{Definition}

\def\qq{\mathbb{Q}}
\def\pp{\mathbb{P}}
\def\rr{\mathbb{R}}
\def\zz{\mathbb{Z}}
\def\cc{\mathbb{C}}
\def\nn{\mathbb{N}}
\def\oo{\mathcal{O}}

\def\jj{\mathcal{J}}

\newcommand{\supp}{\operatorname{supp}}
\renewcommand{\div}{\operatorname{div}}
\newcommand{\Div}{\operatorname{Div}}
\newcommand{\ord}{\operatorname{ord}}

\makeatletter
\@namedef{subjclassname@2010}{%
  \textup{2010} Mathematics Subject Classification}
\makeatother

\numberwithin{equation}{section}

\author{Robin de Jong}
\address{Robin de Jong, Mathematisch Instituut, Universiteit Leiden, PO Box 9512, 2300 RA
Leiden, The Netherlands}
\email{rdejong@math.leidenuniv.nl}

\author{J. Steffen M\"{u}ller}
\address{J. Steffen M\"{u}ller, Fachbereich Mathematik, Universit\"{a}t Hamburg, Bundesstrasse 55, 20146 Hamburg, Germany}
\email{jan.steffen.mueller@uni-hamburg.de}

\begin{document}

\title{Canonical heights and division polynomials}

\begin{abstract}  We discuss a new method to compute the canonical height of an algebraic
    point on a hyperelliptic jacobian over a number field. The method does not require any
    geometrical models, neither $p$-adic nor complex analytic ones. In the case of genus~2
    we also present a version that requires no factorisation at all. 
    The method is based on a recurrence relation for the `division polynomials' associated to hyperelliptic jacobians, and a diophantine approximation result due to Faltings.
\end{abstract}



\maketitle
\thispagestyle{empty}

\section{Introduction}

In \cite{ewNY} G. Everest and T. Ward show how to approximate to high precision the canonical height of an algebraic point on an elliptic curve $E$ over a number field $K$ with a limit formula using the (recurrence) sequence of \emph{division polynomials} $\phi_n$ associated to $E$, and a diophantine approximation result.

The $\phi_n$ have natural analogues for jacobians of hyperelliptic curves. In \cite{uc} Y. Uchida shows how to obtain recurrence relations for the $\phi_n$ for hyperelliptic jacobians of dimension $g \geq 2$. Further there exists a suitable analogue of the diophantine approximation result employed by Everest and Ward, proved by G.~Faltings. In this paper we derive a limit formula for the canonical height of an algebraic point on a hyperelliptic jacobian from these inputs.

We have implemented the resulting method for computing canonical heights in
\verb+Magma+ for $g=2$. The method does not require geometrical models, neither
$p$-adic nor complex analytic ones. If the curve is defined over $\qq$ 
and the coordinates of the point are integral, then it also requires no factorisation. It does need either large integer arithmetic or large $p$-adic and real precision, however. In principle the implementation can be extended to higher genera.

\section{Statement of the main results}

Let $K$ be a number field with ring of integers $\oo_K$ and let $(X, o)$ be a pointed hyperelliptic curve of genus $g \geq 2$ over $K$ given by an
equation $y^2 = f(x)$ with $f \in \oo_K[x]$ monic of odd degree $2g+1$, where $o$ is the
unique point at infinity.
Let $J$ denote the jacobian variety of $X$.
Then the theta divisor $\Theta$ on $J$ is the reduced and irreducible divisor on $J$ whose
support is given by the set of all points which can be represented by a divisor
$(p_1)+\ldots+(p_d) - d(o)$, where all $p_i \in X$ and $d < g$.
Equivalently, these are precisely the points whose reduced Mumford representation $(a(x), b(x))$ (cf. Section~\ref{divpolys})
satisfies $\deg(a) < g$.

For each integer $n \geq 1$ there exists a canonical `division polynomial' $\phi_n$ in the function field of $J$ over $K$, see Section~\ref{divpolys}.
We have
\[ \mathrm{div} \, \phi_n = [n]^*\Theta - n^2 \Theta \, . \]
For each place $v$ of $K$ we further have a canonical local height
function $\widehat{\lambda}_v$, see \cite{uc}, Section~7. These functions are determined by the key relations:
\[
    \log |\phi_n(p)|_v = - \widehat{\lambda}_v (np) + n^2 \widehat{\lambda}_v(p)
\]
for each integer $n\geq 1$, each place $v$ and generic $p \in
J(K_v)$, where $|\cdot|_v$ is the absolute value on $K_v$, normalized
as in Subsection~\ref{global}.

Let $p$ be a point in $J(K)$, not in $\supp(\Theta)$. Let $\widehat{h} \colon J(K) \to \rr$ be the canonical height with respect to the canonical principal polarization on $J$. We have the formula:
\[ [K \colon \qq] \, \widehat{h}(p) = \sum_v n_v \, \widehat{\lambda}_v(p) \, , \]
where $n_v$ is a standard local factor defined in Subsection~\ref{global}.
Put \[T(p) = \{ n \in \zz_{>0} \, | \, np \notin \supp(\Theta) \}.\] Then one can show that $T(p)$ is an infinite set. 

Our first result extends \cite[Theorem~3]{ewNY} and gives a limit formula for
the canonical local height $\widehat{\lambda}_v$ in terms of the division
polynomials.
The proof is based on a diophantine approximation result due to Faltings
(Theorem~\ref{faltings}).
\begin{thm} \label{locallimit} Let $v$ be any place of $K$ and $p \in J(K)
    \setminus \supp(\Theta)$ be a rational point.
Then $T(p)$ is an infinite set and the formula
\[ \widehat{\lambda}_v(p) = \lim_{n \to \infty \atop n \in T(p)} \frac{1}{n^2} \log |\phi_n(p)|_v  \]
holds.
\end{thm}

Let $S$ be a finite set of places of $K$. We put:
\[ \widehat{h}_S(p) = \frac{1}{[K\colon \qq]} \lim_{n \to \infty \atop n \in T(p)} \frac{1}{n^2} \log \prod_{v \in S} |\phi_n(p)|^{n_v}_v \, . \]
Theorem~\ref{locallimit} implies that the limit $\widehat{h}_S(p)$ exists, and gives the $S$-part of the canonical height of $p$.
\begin{thm} \label{limitexists} Assume
that $p$ is a point in $J(K)$, not in $\supp(\Theta)$. Then the limit $\widehat{h}_S(p)$ exists, and the formula
\[ [K:\qq] \, \widehat{h}_S(p) = \sum_{v\in S} n_v \, \widehat{\lambda}_v(p) \]
holds.
\end{thm}
Our next result expresses $\widehat{h}(p)$ in terms of $\widehat{h}_S(p)$ and
$\widehat{h}_S(2p)$, for a suitable set $S$.
Let $\Delta = 2^{4g}\mathrm{disc}(f)$ denote the discriminant of $X$ defined in \cite{lo}.
Then the curve $X$, and hence the jacobian $J$, has good reduction outside
the set $S_\mathrm{bad}$ of places of $K$ dividing the ideal $(\Delta)$.
Let $S_\infty$ be the set of archimedean places of $K$.

\begin{thm} \label{formulaheight}
Let $p\in J(K) $ and assume
that both $p$ and $2p$ are not in $\supp(\Theta)$.
Let $S$ be a finite set of places of $K$ containing $S_\mathrm{bad} \cup S_\infty$, such that for all $v \notin S$ one has that neither $p$ nor $2p$ lies on the theta divisor modulo $v$.
Then the formula
\[ \widehat{h}(p) = -\frac{1}{3} \widehat{h}_S(p) + \frac{1}{3} \widehat{h}_S(2p) \]
holds.
\end{thm}
Note that, for a finite place $v$ outside $S_\mathrm{bad}$, saying that
$p$ lies on the theta divisor modulo $v$ is equivalent to saying that
$p$ can be represented by a divisor $(p_1)+\ldots+(p_g) - g(o)$, where 
one of the $p_i$ reduces to $o$ modulo $v$, or that one of the coefficients of the first
polynomial in the Mumford representation of $p$ is not $v$-integral. 
We will see that $\widehat{h}_S(p)$ and $\widehat{h}_S(2p)$ are effectively computable
for $S$ and $p$ as in Theorem \ref{formulaheight}.

For $g=2$ we can prove a simpler version of Theorem~\ref{formulaheight}.
\begin{thm} \label{formulagenus2}
    Suppose that $g=2$ and that $p\in J(K) \setminus \supp(\Theta)$.
    Let $S$ be a finite set of places of $K$ containing
    $\{v \in S_{\mathrm{bad}}: \ord_v(\Delta) \ge 2\} \cup S_\infty$ such that for all $v \notin
    S$ the point $p$ does not lie on the theta divisor modulo $v$.
    Then we have
    \[
        \widehat{h}(p) = \widehat{h}_S(p).
    \]
\end{thm}
For the proof of Theorem~\ref{formulagenus2}, we compare the canonical local
height $\widehat{\lambda}_v$ to a canonical local height associated with $2\Theta$
introduced by V. Flynn and N. Smart in~\cite{fs}.

The plan of this paper is as follows. Section \ref{canlochts} briefly discusses some basic results around canonical local heights on abelian varieties. In Section \ref{falt} we recall Faltings's diophantine approximation result and deduce a general limit formula from it. After this we focus on hyperelliptic jacobians. First, in Section \ref{divpolys} we review some facts we need from Uchida's paper \cite{uc} on hyperelliptic division polynomials.

Then in Sections \ref{proofsI} and \ref{proofs} we prove Theorems
\ref{locallimit}--\ref{formulagenus2}. Note that, in principle, these results allow one to
approximate values of $\widehat{h}(p)$ effectively. There are two issues to be dealt with.
One is the possible occurrence of large `gaps' in the sets $T(p)$, another is the need to
factor the discriminant in order to apply Theorem \ref{formulaheight}. We discuss, and
resolve to some extent, both issues in Section \ref{gaps}. In particular we can control
the gaps and present a factorisation free approach to computing $\widehat{h}(p)$ in the
genus~$2$ case, adapting an approach described in \cite{ewNY} for the elliptic curves case.

In Section \ref{implementation} we discuss the actual implementation of our method in
\verb+Magma+, and compare our method with earlier ones due to Flynn and Smart \cite{fs},
M. Stoll \cite{st1}, the second author \cite{mu1}, D. Holmes \cite{ho}, and Uchida
\cite{uc1}. We finish the paper by presenting and analysing some data in Section
\ref{examples}. In particular we note that assembling enough data may yield predictions on
the general convergence rate of our limit formulas.

\subsection{Acknowledgements}
We thank Yukihiro Uchida for providing us with formulas for the division polynomials $\phi_n$ for $n\le 5$ when $g=2$.
Some of the research described here was done while the second author was
visiting the University of Leiden and he would like to thank the Mathematical
Institute for its hospitality. The second author was supported by DFG grant KU~2359/2-1.

\section{Canonical local heights} \label{canlochts}

\subsection{Local theory}
We start with some well-known generalities on canonical local heights on abelian
varieties.
See for instance, \cite[Chapter~11]{lfdg}.
\begin{defn} \label{neronfns}
Let $A$ be an abelian variety defined over a local field $K$ with absolute value
$|\cdot|$.
To each divisor $D$ on $A$ one can associate a function
$\lambda_D:A(K)\setminus\supp(D)\to\rr$ such that the following conditions are satisfied.
\begin{enumerate}[\upshape (1)]
 \item If $D,E\in\Div(A)$, then $\lambda_{D+E}=\lambda_{D}+\lambda_{E}+c_1$ for
     some $c_1\in\rr$.
 \item If $D=\div(f)\in\Div(A)$ is principal, then $\lambda_{D}=-\log|f|+c_2$
     for some $c_2\in\rr$.
 \item If $\varphi:A\to A'$ is a morphism of abelian varieties and $D\in\Div(A')$,
       then we have $\lambda_{\varphi^*(D)}=\lambda_{D}\circ \varphi+c_3$  for
       some $c_3\in\rr$.
\end{enumerate}
We call $\lambda_{D}$ a {\em canonical local height (or N\'eron function) associated with $D$}.
\end{defn}
Given a divisor $D$ on an abelian variety defined over a local field $K$, a canonical
local height $\lambda_D$ associated with $D$ is uniquely determined up to a constant.
In particular, if $\lambda_D$ is a canonical local height associated to a symmetric
divisor $D$ on $A$, then by \cite[Proposition~11.1.4]{lfdg}, there exists a
function $\phi \in K(A)^\times$ such that
$\div(\phi) = [2]^* D - 4D$ and
\[
    \lambda_D(2p) - 4\lambda_D(p) = -\log|\phi(p)|
\]
for all $p \in A(K)$ such that both $p$ and $2p$ do not lie in $\supp(D)$.
The function $\phi$ is determined up to a constant factor in $K^\times$ and
$\lambda_D$ is uniquely determined by $\phi$.

Assume now that $K$ is non-archimedean and let $A$ be an abelian variety over $K$.
In this case canonical local heights can be related to the N\'eron model $\mathcal{A}$ of $A$
over the ring of integers $\oo_K$ of $K$.
For $D\in\Div(A)$ and $p\in A(K)$ let
$\overline{D}$ (resp.\ $\overline{p}$) denote the Zariski closures of $D$ with
multiplicities (resp.\ of~ the divisor $(p)$) in $\mathcal{A}$ and let $\lambda_D$ denote a canonical
local height associated with $D$.
Let $v$ denote the closed point of $\mathrm{Spec}(\oo_K)$ and let $i_v(D,p)$ denote the intersection multiplicity
of $\overline{D}$ and~$\overline{p}$ as defined in~\cite[\S 11.5]{lfdg}.

\begin{prop}(N\'eron, cf. \cite[\S 11.5]{lfdg}) \label{iv}
    \begin{enumerate}[\upshape (i)]
        \item If $\mathcal{A}_v$ is connected, then $i_v(D, p)$ is the usual intersection multiplicity of
            $\overline{D}$ and $\overline{p}$ on $\mathcal{A}_v$.
        \item If $\overline{D}$ is represented by $\alpha\in K(\mathcal{A})$ around $\overline{p}
\cap \mathcal{A}_v$, then we have
\[
 i_v(D,p)=-\log|\alpha(p)|.
 \]
\item For each component $\mathcal{C}$ of the special fiber of $\mathcal{A}$ there is a constant
    $\gamma({\mathcal{C}})\in \rr$ such that for all $p\in A(K)\setminus\supp(D)$
    reducing to $\mathcal{C}$ we have
\[
\lambda_{D}(p) = i_v(D,p) + \gamma({\mathcal{C}}). \]
\end{enumerate}
\end{prop}

\subsection{Global theory}\label{global}
Let $K$ be a number field.
There is a standard way of endowing each completion $K_v$ with an absolute value
$|\cdot|_v$, as follows: when $v$ is archimedean, we take the euclidean norm on
$K_v$. When $v$ is non-archimedean, we normalize $|\cdot|_v$ such that
$|\pi|_v=\mathrm{e}^{-1}$, where $\pi$ is a uniformiser of $K_v$. Now let $M_K$ be the set of places of $K$. For each $v \in M_K$ let $n_v$ be the local factor defined as follows: when $v$ is real, then put $n_v=1$; when $v$ is complex, then put $n_v=2$; finally if $v$ is non-archimedean, then $n_v$ is the logarithm of the cardinality of the residue field at $v$. These absolute values fit together in a product formula
$\sum_{v \in M_K} n_v \log |x|_v=0$ valid for all $x$ in $K^\times$.

The connection between canonical heights and canonical local heights is provided by the following result, again due to N\'eron:
\begin{prop}(N\'eron) \label{locglob}
Let $A$ be an abelian variety over $K$ and let $D \in \Div(A)$ be symmetric.
Let $\phi \in K(A)$ such that $\div(\phi) = [2]^* D - 4D$.
For each place $v \in M_K$ we let $\lambda_v$ denote the canonical local height associated
with $D$ on $A(K_v)$ which satisfies
\[
        \lambda_v(2p) - 4\lambda_v(p) = -\log|\phi(p)|_v
\]
for all $p \in A(K_v)$ such that $p$ and $2p$ are not in $\supp(\Theta)$.
Then we have
\[ [K \colon \qq] \, \widehat{h}_D(p) = \sum_v n_v \, {\lambda}_v(p)  \]
for all $p \in A(K) \setminus \supp(\Theta)$,
where $\widehat{h}_D$ is the canonical height associated to $D$.
\end{prop}

\section{Faltings's result and an application} \label{falt}

The following general diophantine approximation result due to G.
Faltings (see \cite{fa}, Theorem II) will be the main ingredient of our method.
\begin{thm} \label{faltings}
Let $A$ be an abelian variety over a number field $K$ and suppose that $D$ is an ample divisor on $A$. Let $v$ be a place of $K$ and let $\lambda_{D,v}$ be a canonical
local height function on $A({K}_v)$ with respect to $D$. Let $h$ be a Weil height
on $A$ associated to some ample line bundle on $A$, and let $k \in \rr_{>0}$ be arbitrary. Then there exist only finitely many points $p \in A(K) \setminus \supp(D)$ such that
$\lambda_{D,v}(p) > k \cdot h(p)$.
\end{thm}
In fact we will use the following corollary.
\begin{thm} \label{faltingscor} Let $A$ be an abelian variety over a number field $K$ and
    let $D$ be a symmetric ample divisor on $A$. Let $v$ be a place of $K$ and let
    $\lambda_{D,v}$ be a canonical local height function on $A({K}_v)$ with respect to
    $D$. Let $p \in A(K) \setminus \supp(D)$ be a rational point and put $T(D,p) = \{ n
    \in \zz_{>0} | np \notin \supp(D) \}$. Then $T(D,p)$ is infinite and we have
    $\lambda_{D,v}(np)/n^2 \to 0$ as $n \to \infty$ over $T(D,p)$.
\end{thm}
\begin{proof} We start by showing that $T(D,p)$ is infinite when $p \notin
    \supp(D)$. For $p$ a torsion point this is immediate. Assume therefore that $p$ is not
    torsion. We prove that for infinitely many $n \in \zz$ we have $np \notin
    \supp(D)$. This is sufficient for our purposes: as $D$ is symmetric, we have
    $np \in \supp(D)$ if and only if $-np \in \supp(D)$. An elementary
    argument on algebraic groups shows that the
    Zariski closure $Z$ of the subgroup $\zz \cdot p$ is a closed algebraic
    subgroup of $A$. Suppose that only finitely many of the $np$ are outside
    $\supp(D)$. Then $Z$ is the union of a finite set with a closed subset of $\supp(D)$. It follows that $Z$ has dimension zero, and hence consists of only finitely many points: contradiction.

The limit formula follows immediately if $p$ is torsion since then the set of values
$\lambda_{D,v}(np)$ as $n$ ranges over $T(D,p)$ is bounded.
Assume therefore that $p$ is not torsion. Then the $np$ with $n$ running
through $T(D,p)$ form an infinite set of $K$-rational points of $A\setminus\supp(D)$.
Let $\widehat{h}$ be the canonical (N\'eron-Tate) height with respect to $D$. Since:
\[ \frac{\lambda_{D,v}(np)}{n^2} = \widehat{h}(p) \cdot
\frac{\lambda_{D,v}(np)}{\widehat{h}(np)} \]
where $\widehat{h}(p)>0$, Theorem \ref{faltings} can be applied, leading to:
\[ \limsup_{n \to \infty \atop n \in T(D,p)}
\frac{\lambda_{D,v}(np)}{n^2} \leq 0 \, . \]
On the other hand, since $\lambda_{D,v}$ is bounded from below we have:
\[ \liminf_{n \to \infty \atop n \in T(D,p)}
\frac{\lambda_{D,v}(np)}{n^2} \geq 0 \, . \]
The theorem follows by combining these two estimates.
\end{proof}
\begin{remark} The above result has the following consequence: let $S$ be a finite set of places of $K$, and assume that $\widehat{h}(p) >0$. Then there is an $N \in \mathbb{N}$ such that for all $n \geq N$,
\[ \sum_{v \notin S} n_v \lambda_{D,v}(np) > 0 \, . \]
It would be interesting to have an effective result in this direction.
\end{remark}

\section{Points and division polynomials} \label{divpolys}
Let $K$ be a field of characteristic not equal to~2 and let $X$ be a hyperelliptic curve of genus $g \geq 2$ over $K$ given by an
equation $y^2 = f(x)$ with $f \in K[x]$ monic of odd degree $2g+1$. We write
$f(x)=\sum_{i=0}^{2g+1} \mu_ix^i$, where $\mu_{2g+1}=1$. Note that $X$ has a unique point $o$ at
infinity. Let $J$ be the jacobian of $X$, endowed with its canonical principal
polarization.
If $p_1 \in X$, then we write $p^-_1$ for the image of $p_1$ under the
hyperelliptic involution.

Then for any point $p \in J$, there is a unique reduced divisor $D = (p_1) +
\ldots + (p_d)$ on $X$ such that $D - d(o)$ represents $p$, which we write as $p = [D -
d(o)]$.
Here we call an effective degree $d$ divisor $D$ on $X$ {\em reduced} if $d \le
g$ and if we have $o \ne p_i \ne p^-_j$ for all $p_i , p_j \in \supp(D)$.
This leads to the {\em Mumford representation}
$ (a(x), b(x))$
of a point $p \in J$: If $(p_1) +\ldots + (p_d)$ is the reduced divisor associated
to $p$, then $a(x) = \prod^d_{i=1} (x - x(p_i)) \in K[x]$ and $b(x) \in K[x]$ is the
uniquely determined polynomial of minimal degree such that $y(p_i) = b(x(p_i))$
for all $i = 1, \ldots, d$.
One also defines the Mumford representation of the origin to be $(1, 0)$.
Note that the map $X^{(g)} \to J$ given by $(p_1,\ldots,p_g) \mapsto
[(p_1)+\cdots+(p_g) - g (o)]$ is birational.

For the construction of the division polynomials $\phi_n$ Uchida uses certain
higher-dimensional generalisations $\wp_{ij}$ and $\wp_{ijk}$, where $i,j,k
\in \{1,\ldots,g\}$, of the Weierstrass $\wp$-function from the theory of
elliptic curves.
Over $\cc$, these functions are constructed as second and third order partial
logarithmic derivatives of the
hyperelliptic $\sigma$-function, respectively.
They are well-defined on the jacobian, see \cite[Proposition~2.5]{uc}.

Despite their analytic construction, the $\wp$-functions make sense over an
arbitrary field of characteristic zero and in fact this continues to hold in
more general situations.
Let $p \in J$, then the values $\wp_{ij}(p)$ and
$\wp_{ijk}(p)$ can be expressed as polynomials in the coefficients of the
Mumford representation $(a(x), b(x))$ of $p$ with coefficients in $\zz[\mu_0,\ldots,\mu_{2g}]$.
More precisely, if we write $a(x) = \sum^g_{i=0} a_ix^i$ and $b(x) =
\sum^{g-1}_{i=0} b_ix^i$, then we have
\begin{equation} \label{pfns}
    \wp_{gj} = -a_{j-1}\qquad \mathrm{and} \qquad \wp_{ggk} = 2b_{k-1}\qquad
\end{equation}
for $j,k \in \{1,\ldots,g\}$ by \cite[Theorem~2.8]{uc}.
Furthermore, the $\wp$-functions $\wp_{gj}$ and $\wp_{ggk}$, where $j,k \in
\{1,\ldots,g\}$, can be used to embed $J \setminus \supp(\Theta)$ into
$\cc^{2g}$. In particular, they have a pole only along $\Theta$.
The other $\wp$-functions can be expressed as polynomials in the $\wp_{gj}$ and
$\wp_{ggk}$ by \cite[Theorem~2.9]{uc}.

The division polynomials $\phi_n$ are also defined in terms of the hyperelliptic
$\sigma$-function and can be expressed as polynomials in terms of the $\wp$-functions
with coefficients in $\zz[1/D,\mu_0,\ldots,\mu_{2g}]$.
Here $D$ is an integer which can be computed
explicitly and is independent of $X$. See \cite[Theorem~5.8]{uc}.
In fact Uchida conjectures \cite[Conjecture~4.14]{uc} that $\phi_n \in
\zz[\mu_0,\ldots,\mu_{2g}][\wp_{ij}, \wp_{ijk}]$ for all $n$.
Moreover, the $\phi_n$ satisfy certain recurrence relations which make it possible to
compute the values they take without the need to construct them as polynomials,
cf. \cite[Theorem~6.4]{uc}.

\section{Proof of Theorems \ref{locallimit} and \ref{limitexists}} \label{proofsI}
Consider the jacobian $J$ of a hyperelliptic curve $X$ of genus $g\ge 2$ defined over a number field $K$, given
by an equation $y^2 = f(x)$, where $f \in \oo_K[x]$ is monic of degree $2g+1$.
Note that every hyperelliptic curve over $K$ of genus $g$ with a $K$-rational Weierstrass point has
such a model.
Let $\Theta$ denote the theta divisor on $J$ with respect to the point $o$ at infinity.
As $-[(p_1)+\ldots+(p_g)-g(o)]=[(p_1^-)+\ldots+(p_g^-)-g(o)]$, we have that $\Theta$ is symmetric.
Recall that for the division polynomial $\phi_2$ we have \[ \div(\phi_2) =[2]^*\Theta-4\Theta.\]
Hence there is a canonical local height function $\widehat{\lambda}_v$
associated with $\Theta$  for each $v \in M_K$ such that
\[
\log |\phi_2(p)|_v = - \widehat{\lambda}_v (2p) + 4 \widehat{\lambda}_v(p)
\]
for  $p \in J(K_v)$ such that $p, 2p \notin \supp(\Theta)$.
Therefore Proposition~\ref{locglob} implies that we have
\[ [K \colon \qq] \, \widehat{h}(p) = \sum_v n_v \, {\widehat{\lambda}}_v(p) \, , \]
where $\widehat{h}$ is the canonical height associated to $\Theta$.

More generally, Uchida shows \cite[Theorem~7.5]{uc} that
\begin{equation} \label{key}
\log |\phi_n(p)|_v = - \widehat{\lambda}_v (np) + n^2 \widehat{\lambda}_v(p)
\end{equation}
for each integer $n\geq 1$ and $p \in J(K_v)$ such that $p, np \notin \supp(\Theta)$.

Using~\eqref{key} and Theorem \ref{faltingscor}, we can prove Theorem~\ref{locallimit}, giving a limit formula for
the canonical local height $\widehat{\lambda}_v$ in terms of the division
polynomials.
\begin{proof}[Proof of Theorem~\ref{locallimit}]
By equation (\ref{key})  we are done once we prove that $T(p)$ is infinite and that  $\widehat{\lambda}_v(np)/n^2 \to 0$ as $n\to \infty$ over $T(p)$. But note that $\widehat{\lambda}_v$ is a canonical local height associated to $\Theta$, which is a symmetric and ample divisor on $J$. The result follows by applying Theorem \ref{faltingscor}.
\end{proof}
The proof of Theorem \ref{limitexists} is now almost immediate.
\begin{proof}[Proof of Theorem \ref{limitexists}]
As $S$ is finite we find:
\begin{align*}
    [K \colon \qq] \widehat{h}_S(p) & = \lim_{n \to \infty \atop n \in T(p)} \frac{1}{n^2} \log \prod_{v \in S} |\phi_n(p)|_v^{n_v} \\
	& = \lim_{n \to \infty \atop n \in T(p)} \frac{1}{n^2} \sum_{v \in S} n_v \log |\phi_n(p)|_v \\
	& = \sum_{v \in S} n_v \lim_{n \to \infty \atop n \in T(p)}
    \frac{1}{n^2} \log |\phi_n(p)|_v.
\end{align*}
By Theorem~\ref{locallimit} we have
\[ \lim_{n \to \infty \atop n \in T(p)}
    \frac{1}{n^2} \log |\phi_n(p)|_v = \widehat{\lambda}_v(p) \]
for each $v \in M_K$. This proves the result.
\end{proof}
\begin{remark}\label{lang-conj} Unfortunately Theorem \ref{faltingscor} does not tell us anything about the convergence rate of the sequence $\left(\frac{1}{n^2}\widehat{\lambda}_v(np) \right)_{n \in T(p)}$ or $\left(\frac{1}{n^2} \log |\phi_n(p)|_v\right)_{n \in T(p)}$.
    If $v$ is archimedean, then a conjecture of
    Lang~\cite[(2.1)]{la1} implies that $\widehat{\lambda}_v(np) = \mathcal{O}(\log n)$.
For elliptic curves, this bound can be proved unconditionally using David and Hirata-Kohno's results on linear forms in elliptic logarithms~\cite{dhk}. 
For non-archimedean $v$, we expect that a more refined analysis of the statements in
Proposition~\ref{iv} will give an $\oo(\log n)$ bound for $\widehat{\lambda}_v(np)$ as
well (in particular one should not need diophantine approximation to prove such a bound). 
\end{remark}

If the genus is~2, then we can compare
$\widehat{\lambda}_v$ to another well-known canonical local height function.
In~\cite{fs}, Flynn and Smart construct a function
$\widehat{\lambda}^{\mathrm{FS}}_v:J(K_v) \to \rr$; Uchida~\cite[Theorem~5.3]{uc1} has
shown that this is a canonical local height
    associated to $2\Theta$ for each place $v$ of $K$.
    Let $\kappa = (\kappa_1,\ldots,\kappa_4): J \to \pp^3$ denote the morphism constructed explicitly
    in~\cite[Chapter~3]{cf}.
    The image of $\kappa$ is the Kummer surface associated to~$J$  embedded
    into~$\pp^3$ and we have $\kappa_1(p) = 0$ if and only if $p \in
    \supp(\Theta)$.
    There are homogeneous quartic polynomials $\delta_i \in
    \zz[\mu_0,\ldots,\mu_4][x_1,\ldots,x_4]$ such that if $p \in J$, then \[\delta(\kappa(p)) = \kappa(2p),\]
    where $\delta = (\delta_1, \ldots, \delta_4)$.
    In addition, the relation $\div(\delta_1\circ \kappa) = [2]^*(2\Theta) - 8\Theta$ holds.

    The canonical local height $\widehat{\lambda}^{\mathrm{FS}}_v$ constructed by Flynn and Smart
    is associated to $2\Theta$ and is determined by the condition that
    \begin{equation} \label{keyfs}
        \widehat{\lambda}^{\mathrm{FS}}_v(2p) -
        4\widehat{\lambda}^{\mathrm{FS}}_v(p) =
        -\log\left|\delta_1\left(\frac{\kappa(p)}{\kappa_1(p)}\right)\right|_v
    \end{equation}
    for all $p \in J(K_v)$ such that both $p$ and $2p$ are not in $\supp(2\Theta)$.
    \begin{prop} \label{uchidavsfs}
        If the genus of $X$ is~2 and if $p \in J(K_v) \setminus
        \supp(\Theta)$, then we have
        \[
            \widehat{\lambda}^{\mathrm{FS}}_v(p) = 2\widehat{\lambda}_v(p).
        \]
    \end{prop}
    \begin{proof}
        Since $\widehat{\lambda}_v$ is a canonical local height associated to
        $\Theta$, it follows from property (i) of Definition~\ref{neronfns} that
        $2\widehat{\lambda}_v$ is a canonical local height associated to
        $2\Theta$.
        Because of~\eqref{key} and~\eqref{keyfs}, it suffices to show that for a
        point $p \in J \setminus \supp(\Theta)$ we have
        \[
            \delta_1\left(\frac{\kappa(p)}{\kappa_1(p)}\right) = \phi_2(p)^2.
        \]
        We have checked this relation symbolically using explicit expressions for $\phi_2$ and $\delta_1$.
For this computation we used the computer algebra system \verb+Magma+ \cite{magma}.
    \end{proof}

\section{Proof of Theorems~\ref{formulaheight}
and~\ref{formulagenus2}} \label{proofs}

In this section we prove Theorems \ref{formulaheight} and~\ref{formulagenus2}.

\begin{proof}[Proof of Theorem \ref{formulaheight}]
For $v \notin S_\mathrm{bad} \cup S_\infty$ the jacobian $J$ has good reduction, so the special fiber
$\jj_v$ of the N\'eron model $\jj$ of $J$ over $\mathrm{Spec}(\oo_K)$ is an
abelian variety.
Hence for such $v$ we have, for all $p$ not in $\supp(\Theta)$, that
$\widehat{\lambda}_v(p)=i_v(p,\Theta)+\gamma_v$ where $i_v$ is the $v$-adic
intersection multiplicity on $\jj$, and $\gamma_v$ is a constant independent of $p$. There are only finitely many $v \notin S$ such that $\gamma_v$ is non-zero. Put $\delta_S=\sum_{v \notin S} n_v \gamma_v$.
The assumption on $p$ implies that for $v \notin S$ we have
$\widehat{\lambda}_v(p)=\widehat{\lambda}_v(2p)=\gamma_v$. We obtain using
Theorem~\ref{limitexists}
\begin{align*} [K:\qq]\, \widehat{h}(p) = & \sum_{v \in S} n_v \widehat{\lambda}_v(p) + \delta_S \\
                                        = & [K:\qq] \, \widehat{h}_S(p) + \delta_S
\end{align*}
and similarly
\[ [K:\qq] \, \widehat{h}(2p) = [K:\qq] \, \widehat{h}_S(2p) + \delta_S \, . \]
Combining this with $\widehat{h}(2p)=4\widehat{h}(p)$ we deduce the required formula.
\end{proof}
\begin{proof}[Proof of Theorem \ref{formulagenus2}]
    Suppose that $g=2$.
    It clearly suffices to show that if $v$ is a finite place of $K$ such that
     $\ord_v(\Delta) \le 1$, then we have
    \begin{equation} \label{suff}
        \widehat{\lambda}_v(p) = i_v(\Theta, p)
    \end{equation}
    for all $p \in J(K_v) \setminus \supp(\Theta)$.

    So let $v$ be such a place.
    It follows from \cite[Proposition~5.2]{st1} that if $p \notin \supp(\Theta)$, then the
    canonical local height $\widehat{\lambda}^{\mathrm{FS}}_v$ constructed by Flynn and Smart
    satisfies
    \begin{equation} \label{fseq}
        \widehat{\lambda}^{\mathrm{FS}}_v(p) = \log \max_{1 \le i \le 4} \left|
        \frac{\kappa_i(p)}{\kappa_1(p)}\right|_v.
    \end{equation}
    Pick integral coordinates $(x_1, \ldots, x_4)$ for
    $\kappa(P)$ in such a way that $x_j$ is a unit for some $j \in \{1, \ldots, 4\}$.
    Then~\eqref{fseq} implies that
    \[
        \widehat{\lambda}^{\mathrm{FS}}_v(p) = -\log \min_{1 \le i \le 4} \left|
    \frac{x_1}{x_i}\right|_v = -\log \left|\frac{x_1}{x_j}\right|_v = -\log |x_1|_v.
    \]
    But since $\kappa_1(p) = 0$ if and only if $p \in
    \supp(\Theta)$, Proposition~\ref{iv} (ii) implies that
    \[
        -\log |x_1|_v = i_v(2\Theta, p) = 2i_v(\Theta, p).
    \]
    Combined with Proposition~\ref{uchidavsfs}, this proves~\eqref{suff} and hence the theorem.
\end{proof}
\begin{remark} The above proof shows that $\gamma_v=0$ if $\ord_v(\Delta) \le 1$
    and $g=2$. For general $g \geq 2$, if $J$ has good reduction at $v$, one has
\[\gamma_v = \frac{-\log|\phi_2(p)|_v}{3}\]
for any $p$ such  that $p$ and $2p$ are not in $\supp(\Theta)$ mod $v$.
This implies that $\gamma_v \ge 0$ for such $v$.
\end{remark}

\section{Gaps and factorisation} \label{gaps}

Suppose now that we want to calculate $\widehat{h}(p)$ for a rational point $p$ on the
jacobian associated to the hyperelliptic curve $X: y^2 = \sum_{i=0}^{2g+1} \mu_i x^i$
defined over a number field $K$, where $g \ge 2$, $\mu_{2g+1}=1$ and all $\mu_i \in \oo_K$.

In order to apply Theorem~\ref{limitexists} or~\ref{formulagenus2}, a
first requirement is that $p$ is not in
$\supp(\Theta)$ (applying Theorem~\ref{formulaheight} requires, in addition, that $2p$ is not in
$\supp(\Theta)$).
If $p \in \supp(\Theta)$, we can simply try to replace $p$ by a multiple.

Next, one wants to know in advance that the set $T(p)$ of multiples to which one is confined does not contain large gaps. Note that a gap of length $g+1$ gives rise to a point in the intersection $\Theta \cap \Theta_{p}\cap \ldots \cap \Theta_{gp}$ of $g+1$ translates of the theta divisor $\Theta$. These translates are distinct if $p$ is not torsion of order $\leq g$, since the morphism $J \to \widehat{J}$ given by $q \mapsto [\Theta - \Theta_q]$ is an isomorphism. Generically one expects the intersection of these translates therefore to be empty.

In the case $g=2$ we can give the following precise statement.
    \begin{lem} \label{gapsgenus2}
    Let $K$ be a field of characteristic not equal to~$2$ and let $X$ be a genus~2 curve defined over
    $K$ with jacobian $J$. Let $p = [(p_1)+ (p_2) - 2(o)] \in J$ be a non-zero point.
    Then we have
\begin{enumerate}[\upshape (i)]
    \item If $p \in J[2]$, then $\bigcap^N_{n=1} \Theta_{np}$ is non-empty
            for all $N\ge 1$.
    \item Assume that neither $p_1$ nor $p_2$ are Weierstrass points. Then $\Theta \cap \Theta_p \cap \Theta_{2p}$ is
            empty.
    \item The intersection $\Theta \cap \Theta_p \cap \Theta_{2p} \cap \Theta_{3p}$ is
            empty for all $p\notin J[2]$.
\end{enumerate}
\end{lem}
\begin{proof} Note that $p$ uniquely determines 
    the unordered pair $\{p_1,p_2\}$ by Riemann-Roch.
    If $p \in J[2] \setminus \{0\}$, then both
    $p_1$ and $p_2$ are Weierstrass points. One then readily checks that in this situation both $[(p_1)-(o)]$ and
$[(p_2)-(o)]$ lie in $\Theta \cap \Theta_p$, which proves~(i).

    Now let $p \in J \setminus \{0\}$ be arbitrary and suppose
$q=[(q_1)- (o)]\in \Theta \cap \Theta_p$.
Then there exists $r=[(r_1)- (o)] \in \Theta$ such that $p=r-q$ and hence
\[
    (p_1) + (p_2) - 2(o) \sim (r_1) - (q_1).
\]
By Riemann-Roch this implies
\begin{equation} \label{rqp-equation}
    (r_1) - (q_1) \in \{(p_1) - (p^-_2), (p_2) - (p^-_1)\}
\end{equation}
and hence $q_1=p_1^-$ or $q_1=p_2^-$.
Without loss of generality we assume that $q_1 = p^-_1$.

Suppose that $q \in \Theta \cap \Theta_p \cap \Theta_s$ where
$s = 2p = [(s_1)+ (s_2)-2(o)]$. Similarly as before we find that $q_1=s_1^-$ or $q_1=s_2^-$. Hence $s_i = p_1$ for some $i
\in \{1,2\}$, say $s_1 = p_1$.
This implies
\[
    p = s - p =  [(s_2) - (p_2)].
\]
Again by Riemann-Roch we find
\[
    (s_2) - (p_2) \in \{(p_1) - (p^-_2), (p_2) - (p^-_1)\}
\]
leading to $p_2 = p^-_1$ or $p_2 = p^-_2$. The first possibility implies that $p=0$, which we excluded, so we end up with $p_2=p_2^-$. This proves (ii).

To prove (iii), we may assume that $p_2 = p^-_2$, so that $2p = [2(p_1)-2(o)]$. 
Note that under this assumption $p \notin J[3]$, since otherwise we would have 
$2p = -p$, which implies $p_1 = p^-_1$ or $p_1 = p_2^-$, and hence $p \in J[2]\cap J[3] =
\{0\}$.

By the arguments above, we may assume that a point $q \in \Theta \cap \Theta_p \cap \Theta_{2p}$
satisfies $q_1 = p_1^-$.
If we assume, in addition, that $q \in \Theta_{t}$, where $t = [(t_1)+(t_2)]
=3p \ne 0$, then Riemann-Roch implies $p_1 \in \{t_1, t_2\}$ as
in~\eqref{rqp-equation}, say $p_1 = t_1$.
But then
\[
    p = 3p-2p = [(t_1) + (t_2) - 2(p_1)] = [(t_2) - (p_1)] 
\]
which implies $p \in J[2]$.
\end{proof}

Note that we also need to find the primes dividing the ideal $(\Delta)$ if we want to apply
Theorem~\ref{limitexists} or~\ref{formulagenus2}.
In practice, this becomes problematic if $N_{K/\qq}(\Delta)$ is large.
The following result generalizes equation~(21) in \cite{ewNY}.
\begin{thm} \label{practical} 
    Assume that $X$ is defined over $\qq$.
    Let $p \in J(\qq) \setminus \supp(\Theta)$ such that
$\phi_n(p) \in \zz$ for all $n\ \ge 1$ and put $E_n= \phi_n(p)$.
Let $S'$ be a finite set of primes of $\qq$ containing $S_\mathrm{bad}$ and write $S=S'\cup\{\infty\}$. 
 Assume that $l$ is a positive integer such that for all reductions $\widetilde{J}$
 of $J$ modulo primes not in $S'$
 we have that $T(\widetilde{p})$ contains no gap larger than $l$, where
 $\widetilde{p}$ is the reduction of $p$.
 Then we have
 \[\widehat{h}_S(p) = \lim_{n \to \infty \atop n
 \in T(p)} \frac{1}{n^2} \log \left(
 \frac{|E_n|}{\gcd(|E_n|,|E_{n+1}|,\ldots,|E_{n+l}|)} \right). \]

 \end{thm}
 \begin{proof}
Note that
\[ \prod_{v \in S} |\phi_n(p)|^{n_v}_v = |E_n|\prod_{v \in S'} |E_n|^{n_v}_v \]
hence
\[ \widehat{h}_S(p) = \lim_{n \to \infty \atop n \in T(p)} \frac{1}{n^2} \log
|E_n|\prod_{v \in S'} |E_n|^{n_v}_v  \, . \]
By assumption, we have that for each given $n \in T(p)$ a prime $v \notin S'$ does not
occur in all of $E_n,\ldots,E_{n+l}$ simultaneously, so that the gcd is only
composed of primes in $S'$.
In fact we have
\begin{align*} \gcd(|E_n|,\ldots,|E_{n+l}|) = & \prod_{v \in S'} \min(|E_n|_v^{-1},\ldots,|E_{n+l}|_v^{-1})^{n_v} \\
= & \prod_{v \in S'} |E_n|_v^{-n_v} \min(1,|E_{n+1}/E_n|_v^{-1},\ldots,|E_{n+l}/E_n|_v^{-1})^{n_v} \, .
\end{align*}
Thus it suffices to show that in the limit as $n \to \infty$ one has
\begin{equation} \label{gcdsequence}
    \frac{1}{n^2}\log\min(1,|E_{n+1}/E_n|_v^{-1},\ldots,|E_{n+l}/E_n|_v^{-1}) \to
    0
\end{equation}
for $n \in T(p)$.
By Theorem~\ref{locallimit}, the sequence $\left(n^{-2} \log|E_n|_v\right)_{n\in T(p)}$
converges for every $v \in S'$, hence is a Cauchy sequence.
This proves~\eqref{gcdsequence} and therefore the theorem.
 \end{proof}
 Using Theorem~\ref{practical} and Lemma~\ref{gapsgenus2}, we can develop a method for the computation of
$\widehat{h}(p)$ if $K=\qq$ and $g=2$ which requires no factorisation at all.
 \begin{cor}\label{g2method}
     Suppose that $g=2$ and that $p  \in J(\qq)$ satisfies
     $\wp_{2j}(p),\wp_{22k}(p) \in \zz$ for $j,k \in \{1,2\}$, and 
     $\gcd(a(x), b(x)) = 1$, where 
     $(a(x), b(x))$ is the Mumford representation of $p$.
      Suppose, moreover, that $\phi_n(p) \in \zz$ for all $n \ge 1$.
 Then we have
 \[\widehat{h}(p) =  \lim_{n \to \infty \atop n
 \in T(p)} \frac{1}{n^2} \log \left(
 \frac{|E_n|}{\gcd(|E_n|,|E_{n+1}|,|E_{n+2}|)} \right). \]
 \end{cor}
 \begin{proof}
 Write $p =[(p_1)+( p_2)-2(o)]$, where both $p_1,p_2 \in X(K)$ and $K= \qq$ or $K$ is a
 quadratic extension of $\qq$. The condition $\gcd(a(x), b(x)) = 1$ ensures that neither $p_1$ nor $p_2$ is a Weierstrass point on $X$. In order to apply Theorem~\ref{practical} we 
 let $S'$ denote the union of $S_\mathrm{bad}$ and the finite set of places $v$ such that
 $p_1$ or $p_2$ reduces to a Weierstrass point modulo $w$ for some place $w$ of $K$
 dividing $v$. 
By Lemma~\ref{gapsgenus2} we can then take $l=2$. Put $S=S'\cup \{\infty\}$. 

By Theorem~\ref{practical}, the right hand side of the equality to be proven equals $\widehat{h}_S(p)$. The assumptions that $\wp_{2j}(p),\wp_{22k}(p) \in \zz$ for $j,k \in \{1,2\}$ imply that for $v \notin S$ the point $p$ does not lie on the theta divisor modulo $v$. The equality itself then follows by applying Theorem~\ref{formulagenus2}.
  \end{proof}
 \begin{remark} \label{integrality}
Assuming that all the $\phi_n(p)$ are integers may seem like a strong restriction,
but, possibly after applying a simple coordinate transformation to $X$, we can at least 
always assume that all $\wp_{gj}(p)$ and $\wp_{ggk}(p)$ are integral.
Then, a conjecture of Uchida \cite[Conjecture~4.14]{uc} predicts that all $\phi_n(p)$ are
integral.
So we can simply test
along the way  whether $E_n$ has a nontrivial denominator for
$n=1,2,\ldots$; such an $n$ would then yield a counterexample
to~\cite[Conjecture~4.14]{uc}.
\end{remark}
\begin{remark} We note that if $p =[(x_1,y_1)+\ldots+(x_g,y_g)-g(o)]\in J(\qq)$ such that all $x_i$ and $y_i$
are integral, then all $\wp_{gj}(p)$ are integral, but this need not hold for all
$\wp_{ggk}(p)$.
Consider, for instance, the Jacobian $J$ of the hyperelliptic curve $X$ given by the affine model
\[
y^2 = 1+2x+3x^2+4x^3+5x^4+x^5
\]
and the point $p=[(1,4)+ (-2,5)-2(o)] \in J$, satisfying
\[
 \wp_{21}(p)= -1,\,
 \wp_{22}(p)= 2,\,
 \wp_{221}(p)= -2/3,\,
 \wp_{222}(p)= 26/3.
\]
\end{remark}

\section{Implementation}\label{implementation}
Suppose that $g=2$.
We have implemented the computation of the values of $\phi_n$ for this case in
\verb+Magma+.
Expressions for the $\wp$-functions $\wp_{11},\,\wp_{112}$ and $\wp_{111}$ 
in terms of $\wp_{12}, \wp_{22}, \wp_{122}$ and $\wp_{222}$ are given in
\cite[Example~5.9]{uc}.
Uchida shows that all $\phi_n \in \zz[1/2, \mu_0,\ldots,\mu_4]$ and conjectures
that in fact $\phi_n \in \zz[\mu_0,\ldots,\mu_4]$.
The division polynomials $\phi_n$ for $n \in \{1,\ldots,5\}$ were already
computed by Uchida and we are grateful to him for sharing them with us.
In fact it is not hard to compute these using a method already discussed by
Kanayama~\cite{ka1} who first constructed the division polynomials in the genus~2
case.

We have not computed any of the $\phi_n$ for $n > 5$ as polynomials because they
quickly become rather complicated.
Instead we employ a recurrence relation due to Kanayama \cite[Theorem~9
(corrected)]{ka2} which can be used to compute $\phi_{2n+1}$ ($n \ge 2$) and $\phi_{2n}$ ($n \ge 3$)
in terms of $\phi_{n-2},\ldots,\phi_{n+2}$ and some of their partial derivatives.
Given $p \in J(\qq) \setminus \supp(\Theta)$, we apply this method for the calculation of $\phi_n(p)$, where $n = 6,7,8$; our
method relies on finding partial derivatives of $\phi_2, \ldots, \phi_5$ for
our specific $J$ and then evaluating them at $p$.

Having determined $\phi_1(p), \ldots, \phi_8(p)$, we then proceed to use
Uchida's recurrence relations from \cite[Example~6.6]{uc} to compute $\phi_n(p)$
for $n \ge 9$.
These are preferable to Kanayama's recurrence relations since they only need
the values $\phi_m(p)$ for $m \in \{1,\ldots,5\}$ and $m \in
\{\frac{n-7}{2},\ldots, \frac{n+7}{2}\}$ (resp. $m \in
\{\frac{n-8}{2},\ldots, \frac{n+8}{2}\}$) if $n$ is odd (resp. even); no derivation of polynomials is required.

We have implemented the computation of $\widehat{h}(p)$ using both
Theorem~\ref{formulagenus2} and Corollary~\ref{g2method}.
If we can factor $\Delta$, then it is usually much faster to
use Theorem~\ref{formulagenus2} and work locally at each relevant place.
The code is available on the second author's homepage
{\url{http://www.math.uni-hamburg.de/home/js.mueller/#code}}.

Several other methods exist for the computation of canonical heights on hyperelliptic
jacobians.
For instance, Holmes~\cite{ho} and the second author~\cite{mu1} have independently
developed algorithms that can be used for arbitrary $g \ge 1$; the current record
computation has $g=10$, see \cite[\S6]{mu1}. 
Their methods need integer factorisation, regular models of the curves and theta functions
on $\cc^g$. 

For $g=2$ other algorithms are available.
These all require explicit arithmetic on a model of the Kummer surface associated to $J$
in $\pp^3$, see Section~\ref{proofsI}.
The original method of Flynn and Smart~\cite{fs} requires no integer factorisation, but
needs the computation of a certain multiple $np$ of the point $p \in J(K)$ whose canonical height
we want to compute.
As $n$ can become quite large (see~\cite[\S1]{st1}), this often becomes impractical.
A modified version due to Stoll~\cite{st1} remedies this, but requires integer
factorisation. However, one can combine this modified version with the original method of
Flynn and Smart to avoid difficult factorisations, see~\cite[\S6]{st1}.
Further improvements are given in~\cite{ms}.
Another algorithm which is very similar to Stoll's method is due to Uchida~\cite{uc1}.
One could extend these techniques to higher genus if one had formulas for explicit
arithmetic on a model of the Kummer variety. 
This is already quite difficult in genus~3, see for instance~\cite{mu2}; Stoll has
recently found an analogue of his genus~2 algorithm in genus~3~\cite{st2}

Currently, {\tt Magma} contains an implementation of the algorithms from~\cite{mu1} for
general $g$ and~\cite{st1} for $g=2$.
When $g=2$, then the algorithm from~\cite{st1} is usually faster than the algorithms using
Theorem~\ref{formulagenus2} or Corollary~\ref{g2method}, which in turn are usually faster than the
implementation of the algorithm from~\cite{mu1} if we are only interested in a few digits
of precision.
\section{Examples}\label{examples}
\subsection{Height computation}
Let $X$ be given by the affine model
\[
y^2 = 1+2x+3x^2+4x^3+5x^4+x^5
\]
and let $J$ be the Jacobian of $X$.
We want to compute the canonical height $\widehat{h}(p)$ of the
point $p=[(1,4)+ (-2,-5)-2(o)] \in J$, satisfying
\[
 \wp_{21}(p)= -1,\,
 \wp_{22}(p)= 2,\,
 \wp_{221}(p)= 6,\,
 \wp_{222}(p)= 2.
\]
Using the implementation of the Flynn-Smart algorithm~\cite{fs} modified by
Stoll~\cite{st1} in \verb+Magma+, we compute $\widehat{h}(p) \sim 0.905661971737515301104367671719 $.

We can use Corollary~\ref{g2method} to compute $\widehat{h}(p)$ without any factorisations,
see Table~\ref{globalcomp}.
If we are only interested in a few digits of precision, it suffices to compute $\phi_n(p)
$ for $n\le 100$. In this case the bulk of the computation is spent on the computation of
$\phi_n(p)$ for $n\le8$, because, as mentioned in Section~\ref{implementation}, we need to
manipulate polynomials.
For the computation of $\phi_n(p)$ for $n \ge 9$ recurrence relations are used which only
need the values $\phi_m(p)$ for a few $m<n$, see Section~\ref{implementation}.

If we are interested in more than~4 digits of precision, then the computation of
$\widehat{h}(p)$ using Theorem~\ref{formulagenus2} is much faster, see Table~\ref{localcomp}.
The prime factorisation of the  discriminant of $X$ is $\Delta=2^8\cdot 86477$, so it suffices to consider the set of
places $S= \{2, \infty\}$, since $p$ has integral
$\wp_{2j}(p),\,\wp_{22k}(p)$.

\begin{center}
\begin{table}\begin{tabular}{|l|l|l|}
\hline
Iterations & Running time in seconds & Error \\
\hline
10 & 0.33 & $3.60\cdot 10^{-2}$ \\
100 & 0.36 & $4.67\cdot 10^{-4}$ \\
200 & 0.74 & $1.27\cdot 10^{-4}$  \\
300 & 2.60 & $6.92\cdot 10^{-5}$  \\
400 & 7.73 & $3.49\cdot 10^{-5}$  \\
500 & 18.990 & $2.45\cdot 10^{-5}$  \\
\hline
\end{tabular}
\caption{$\;\;$Computing $\widehat{h}(p)$ using Corollary~\ref{g2method}}\label{globalcomp}\end{table}
\end{center}
\begin{center}
\begin{table}\begin{tabular}{|l|l|l|}
\hline
Iterations & Running time in seconds & Error \\
\hline
10 & 0.72 & $3.60\cdot 10^{-2}$ \\
100 & 0.74 & $4.67\cdot 10^{-4}$ \\
1000 & 0.89 & $ 4.82\cdot 10^{-6}$ \\
5000 & 1.58 & $1.93\cdot 10^{-7}$  \\
10000 & 2.45 & $5.86\cdot 10^{-8}$  \\
15000 & 3.30 & $2.26\cdot 10^{-8}$  \\
20000 & 4.14 & $1.65\cdot 10^{-8}$  \\
25000 & 4.96 & $1.21\cdot 10^{-8}$  \\
\hline
\end{tabular}
\caption{$\;\;$Computing $\widehat{h}(p)$ using Theorem~\ref{formulagenus2}}\label{localcomp}\end{table}
\end{center}

\subsection{Order of growth of $\widehat{\lambda}_v(np)$}
As was remarked before, just by using Faltings's Theorem~\ref{faltings} we are not able to say
anything about the convergence rate of the sequence
$\left(\frac{1}{n^2} \log |\phi_n(p)|_v\right)_{n \in T(p)}$ for a given place $v$.
By~\eqref{key}, finding this convergence rate is equivalent to finding the order of growth of $\widehat{\lambda}_v(np)$.

We have applied our implementation described in Section~\ref{implementation} to gather data 
on the asymptotic behaviour and the implied constants of the
sequence $(\widehat{\lambda}_{v}(np))_{n\in\nn}$, where $p \in J(\qq)$ is a rational point on a genus~2 jacobian and $v \in M_\qq$.
To this end we varied the place $v$, the coefficients $\mu_i$ and the point $p$.
More precisely, we considered about~2000 random genus~2 curves with $|\mu_i| \le 50$ for
$i \in \{1,\ldots,4\}$; we
computed $\widehat{\lambda}_{v}(np)$ for $v=\infty$ and all non-archimedean $v$ such that $\ord_v(\Delta) \ge
2$, for all $p \notin \supp(\Theta)\cap J[2]$ of Kummer surface height bounded
by~500 and for all $n \in \{1,\ldots,15000\}\cap T(p)$.
We also considered about~100 examples of curves with $50<|\mu_i|\le1000$.
\subsubsection{Archimedean places}
Let us first describe the case $v=\infty$.
As mentioned in Remark~\ref{lang-conj}, by a conjecture of Lang we should have
\[
\widehat{\lambda}_\infty(np) = \oo(\log n)
\]
for $n \in T(p)$. We have used our implementation to test this prediction.

See Figure~\ref{arch-growth} for the values of $\widehat{\lambda}_\infty(np)$, where
$n \in \{1,\ldots,15000\}$ and $p \in J_1(\qq)$ has Mumford representation
\[
(x^2 + 1081/25x + 148/5, 13803/125x + 1799/25).
\]
Note that every $n \in \{1,\ldots,15000\} $ lies in $T(p)$.
Here $J_1$ is the jacobian of the genus~2 curve given by
\[
y^2 = 25+20x+30x^2+40x^3+50x^4+x^5.
\]
All examples we have considered exhibit a similar behavior. The resulting data suggest that we may even have
\[
    \widehat{\lambda}_\infty(np) = \oo((\log n)^A)
\]
for some $0<A<1$ depending on $X$ and $p$, and that the implied constant is rather small compared to the coefficients $\mu_i$.

\begin{figure}
 \begin{center}
    \includegraphics[width=\textwidth]{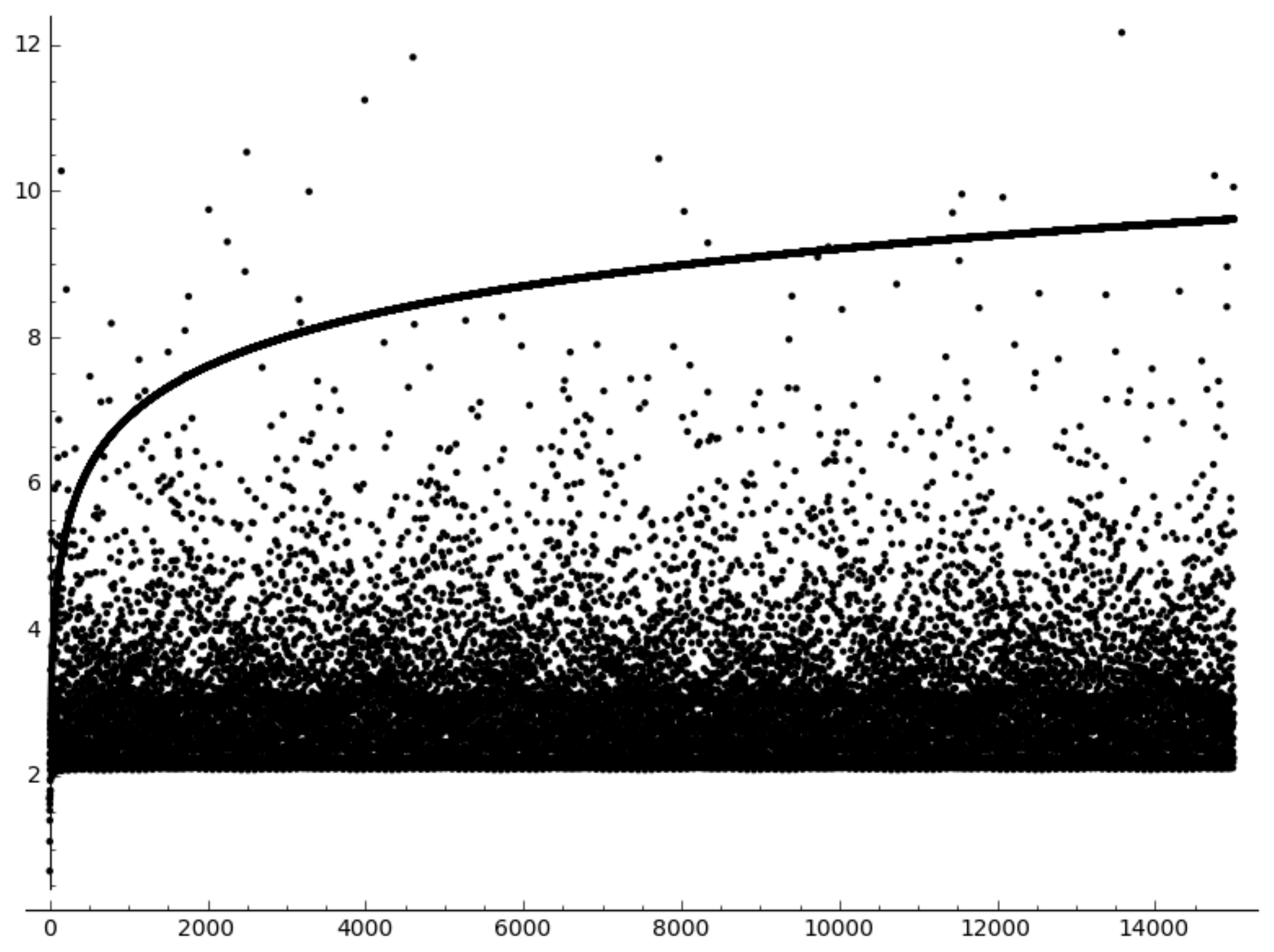}
    \caption{$\widehat{\lambda}_\infty(np)$ and $\log(n)$ for $n \in\{1,\ldots,15000\}$}\label{arch-growth}
  \end{center}
\end{figure}

\subsubsection{Non-archimedean places}
Let $J_2$ be the
jacobian of the genus~2 curve given by
\[
y^2 = 100+200x+300x^2+400x^3+500x^4+x^5
\]
and let $q \in J_2(\qq)$ have Mumford representation
\[
(x^2 + 400x + 200, 3990x + 1990).
\]
Then $q$ reduces to a singular point on the reduction of $J_2$
modulo $v=2$;
the values of $\widehat{\lambda}_2(nq)$ are shown in
Figure~\ref{na-growth}.

Note the apparent formation of finitely many horizontal lines, as well as a set 
of `sporadic' points following the graph of $\log n$. This dual behavior can perhaps be
explained using Proposition \ref{iv}~(ii) and (iii) as follows: the set of specialisations
$n\tilde{p}_2$ of the $nq$ in the special fiber of the N\'eron model modulo $v$ is a
finite group~$R$. The group $R$ has a partition $R = R_1 \sqcup R_2$ into points which are
on resp. off the closure of the theta divisor modulo $v$. The values of
$\widehat{\lambda}_2(nq)$ display a $\log n$ behavior for $n\tilde{p}_2 \in R_1$, and are given by $\gamma(\mathcal{C})$, with $\mathcal{C}$  the component containing $n\tilde{p}_2$, when $n\tilde{p}_2 \in R_2$.
Again, a similar behaviour occurred in all our examples.
\begin{figure}[t]
    \begin{center}
   \includegraphics[width=\textwidth]{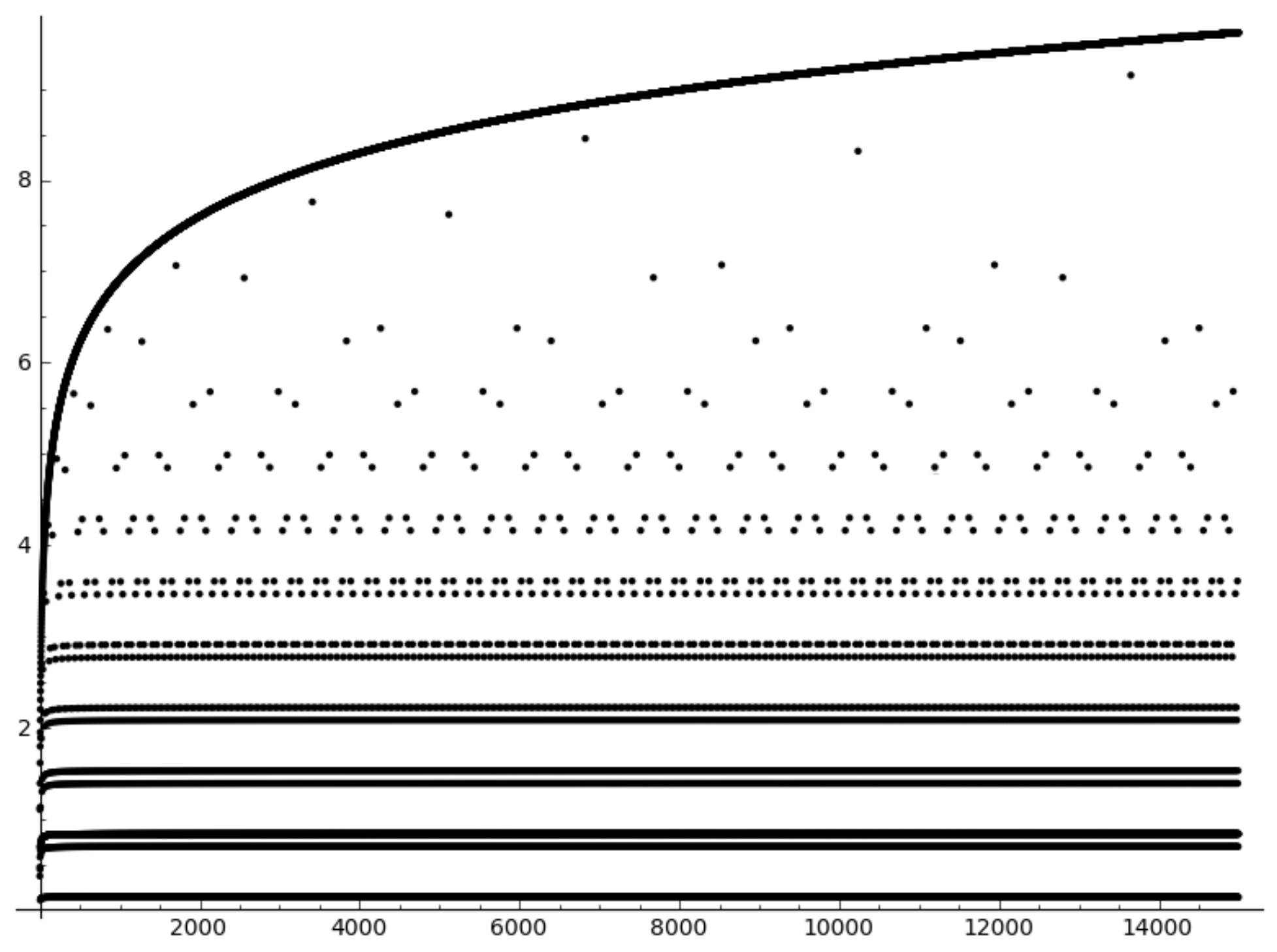}
   \caption{$\hat{\lambda}_2(nq)$ and $\log(n)$ for $n \in\{1,\ldots,15000\}$}\label{na-growth}
    \end{center}
\end{figure}

\end{document}